\documentclass[reqno]{amsart}
\usepackage[latin1]{inputenc}
\usepackage[pdftex]{graphicx}
\usepackage{graphics}
\usepackage{graphicx}
\usepackage{epstopdf}
\usepackage[pdftex,hyperindex]{hyperref}
\usepackage{floatflt}
\usepackage{multicol}
\usepackage[T1]{fontenc}
\usepackage{textcomp}
\usepackage{latexsym}
\usepackage{expdlist}
\usepackage{vmargin}
\usepackage{booktabs}
\usepackage{placeins}
\usepackage{array}
\usepackage{amsmath}
\usepackage{amssymb}
\usepackage{amsfonts}
\usepackage{verbatim}
\usepackage{array}
\usepackage{framed}
\usepackage{amsthm}
\usepackage{enumerate}
\usepackage{cite}
\input amssym.def
\input amssym
\theoremstyle{plain}
\newtheorem{theorem}{Theorem}[section]
\newtheorem{lemma}{Lemma}[section]
\newtheorem{corollary}{Corollary}[section]

\theoremstyle{remark}
\newtheorem{remark}{Remark}[section]

\theoremstyle{definition}

\newtheorem{example}{Example}[section]

\allowdisplaybreaks[1]

\begin{document}

\title[Tur\'{a}n's inequality]{On Tur\'{a}n's inequality: new general criteria, nonnegative representations and the class of generalized Chebyshev polynomials}

\author{Stefan Kahler}

\address{Fachgruppe Mathematik, RWTH Aachen University, Pontdriesch 14-16, 52062 Aachen, Germany}

\email{kahler@mathematik.rwth-aachen.de}

\date{\today}

\begin{abstract}
Originally, Tur\'{a}n's inequality states that if $(P_n(x))_{n\in\mathbb{N}_0}$ is the sequence of Legendre polynomials, then $\Delta_n(x):=P_n^2(x)-P_{n+1}(x)P_{n-1}(x)\geq0$ for all $n\in\mathbb{N}$ and $x\in[-1,1]$. Gasper specified the parameters $\alpha,\beta>-1$ for which the Jacobi polynomials $(R_n^{(\alpha,\beta)}(x))_{n\in\mathbb{N}_0}$ satisfy Tur\'{a}n's inequality. Frequently, such results rely on the specific structure of the concrete orthogonal polynomials under consideration. Therefore, special focus has been put on general criteria (whose importance was particularly emphasized by Nevai). We provide two general criteria for Tur\'{a}n's inequality in terms of the three-term recurrence relation and also deal with sharper estimations of the Tur\'{a}n determinants $\Delta_n(x)$. They extend earlier results of Szwarc and Berg--Szwarc. Applying our criteria to the class of generalized Chebyshev polynomials $(T_n^{(\alpha,\beta)}(x))_{n\in\mathbb{N}_0}$, which are the quadratic transformations of the Jacobi polynomials, we find the companion to Gasper's above-mentioned result. At this stage, we also obtain nonnegative representations of $\Delta_n(x)$. Finally, we study $2$-sieved polynomials and discuss further examples.
\end{abstract}



\maketitle

\numberwithin{equation}{section}

\section{Introduction}\label{sec:intro}

Let $(T_n(x))_{n\in\mathbb{N}_0}$ be the sequence of Chebyshev polynomials of the first kind, i.e., $T_n(\cos(\theta))=\cos(n\theta)$. As a trivial (and well-known) result from elementary trigonometry, one has $T_n^2(\cos(\theta))-T_{n+1}(\cos(\theta))T_{n-1}(\cos(\theta))=\sin^2(\theta)$ and therefore $T_n^2(x)-T_{n+1}(x)T_{n-1}(x)=1-x^2\geq0$ for all $n\in\mathbb{N}$ and $x\in[-1,1]$, with equality if and only if $x=\pm1$. For more general orthogonal polynomial sequences $(P_n(x))_{n\in\mathbb{N}_0}\subseteq\mathbb{R}[x]$, the `Tur\'{a}n determinants'
\begin{equation*}
\Delta_n(x):=P_n^2(x)-P_{n+1}(x)P_{n-1}(x)
\end{equation*}
are much more nontrivial. Moreover, `Tur\'{a}n's inequality', i.e.,
\begin{equation*}
\Delta_n(x)\geq0\;(n\in\mathbb{N},x\in[-1,1]),\;\Delta_n(x)>0\;(n\in\mathbb{N},x\in(-1,1)),
\end{equation*}
may be satisfied or may fail. Frequently, explicit computations are cumbersome or out of reach. The literature on Tur\'{a}n determinants and Tur\'{a}n's inequality is extensive. The topic is interesting and vivid for several reasons. One of these reasons is that Tur\'{a}n determinants have certain relations to the orthogonalization measure due to convergence properties \cite{BS09,FLS04,GvA91,MN83,MNT87,Os22,ST20,ST23,ST24,vA90}. Other reasons are the study of extreme zeros \cite{Kr05,Kr11} and lower bound estimates \cite{BS09}. There are various generalizations of Tur\'{a}n's inequality. A relation to the Riemann hypothesis can be found in \cite{CNV86}. Moreover, the recent publication \cite{HN24} shows relations to combinatorics and number theory.\\

If not stated otherwise, we assume that $(P_n(x))_{n\in\mathbb{N}_0}$ is given by a recurrence relation of the form $P_0(x)=1$,
\begin{equation}\label{eq:threetermrec}
x P_n(x)=a_n P_{n+1}(x)+c_n P_{n-1}(x)\;(n\in\mathbb{N}_0)
\end{equation}
with $c_0:=0$, $(c_n)_{n\in\mathbb{N}}\subseteq(0,1)$ and $a_n\equiv1-c_n$.\footnote{We make the usual convention that $0$ times (or divided by) something undefined (like $P_{-1}(x)$) shall be interpreted as $0$.} It is well-known from the theory of orthogonal polynomials that this is just equivalent to $(P_n(x))_{n\in\mathbb{N}_0}$ being normalized by $P_n(1)\equiv1$ and orthogonal w.r.t. a (unique) symmetric probability (Borel) measure $\mu$ on $\mathbb{R}$ with $|\mathrm{supp}\;\mu|=\infty$ and $\mathrm{supp}\;\mu\subseteq[-1,1]$ (cf. \cite{LO08}), and it is also well-known that the zeros of the polynomials are real, simple and located in the interior of the convex hull of $\mathrm{supp}\;\mu$\cite{Ch78}. In the literature, such sequences $(P_n(x))_{n\in\mathbb{N}_0}$ are also called (symmetric) `random walk polynomial sequences' due to their relation to random walks \cite{CSvD98,KM59,vDS93}.\\

If $(P_n(x))_{n\in\mathbb{N}_0}$ equals the Legendre polynomials, then Tur\'{a}n's inequality is satisfied \cite{Sz48,Tu50}.\footnote{This result, which is due to Tur\'{a}n and was communicated by him to Szeg{\H{o}} in a letter \cite{Tu50}, is the origin of the terminology ``Tur\'{a}n's inequality''. A historic overview can be found in \cite{BS09,Sz98}.} A remarkable (but also around 75 years old) extension concerns the (renormalized) ultraspherical polynomials, which are orthogonal w.r.t. $(1-x^2)^{\alpha}\chi_{(-1,1)}(x)\,\mathrm{d}x$ with $\alpha>-1$: independently from $\alpha$, Tur\'{a}n's inequality is always satisfied \cite{Da55,Ni13,Sk54,Sz50,Sz51,Sz48,TN51} (interesting refinements can be found in \cite{Ni23,NP15}). The above-mentioned Chebyshev polynomials of the first kind correspond to the special case $\alpha=-1/2$, and the Legendre polynomials correspond to the special case $\alpha=0$. An even more remarkable extension concerns the Jacobi polynomials $(R_n^{(\alpha,\beta)}(x))_{n\in\mathbb{N}_0}$, which are orthogonal w.r.t. $(1-x)^{\alpha}(1+x)^{\beta}\chi_{(-1,1)}(x)\,\mathrm{d}x$ with $\alpha,\beta>-1$ and shall also be normalized by $R_n^{(\alpha,\beta)}(1)\equiv1$ (note that these polynomials are generally not symmetric and hence not covered by our general assumptions on $(P_n(x))_{n\in\mathbb{N}_0}$; however, the definition of $\Delta_n(x)$ shall just be as above). In \cite{Ga72}, Gasper showed that the Jacobi polynomials satisfy Tur\'{a}n's inequality if and only if $\alpha\leq\beta$. After earlier contributions \cite{Sz62,Ga71b}, this nontrivial result finally solved a problem posed by Karlin and Szeg{\H{o}} \cite{KS6061}.\\

Despite its shortness and elegance, Gasper's proof has the disadvantage that it makes heavily use of the specific structure of the Jacobi polynomials, particularly concerning their differential equation and forward shift. This makes it difficult to generalize the approach to other classes. Turning back to general (symmetric) orthogonal polynomials $(P_n(x))_{n\in\mathbb{N}_0}$ as above, we recall a valuable general criterion from \cite[Theorem 1]{Sz98}:

\begin{theorem}\label{thm:szwarcturan}
\begin{enumerate}[(i)]
\item If $(c_n)_{n\in\mathbb{N}}\subseteq(0,1/2]$ is nondecreasing, then the polynomials $(P_n(x))_{n\in\mathbb{N}_0}$ satisfy Tur\'{a}n's inequality.
\item If $(c_n)_{n\in\mathbb{N}}\subseteq[1/2,1)$ is nonincreasing, then the polynomials $(P_n(x))_{n\in\mathbb{N}_0}$ satisfy Tur\'{a}n's inequality.
\end{enumerate}
\end{theorem}

We mention that in some way Tur\'{a}n's inequality for the chosen normalization $P_n(1)\equiv1$ is particularly ``strong'' compared to other normalizations, cf. \cite[Proposition 2]{Sz98} and \cite{BS09}. As a trivial fact, we recall that our general assumptions on $(P_n(x))_{n\in\mathbb{N}_0}$ yield that the Tur\'{a}n determinants vanish at $x=\pm1$. Theorem~\ref{thm:szwarcturan} covers the ultraspherical polynomials because for $\alpha\geq-1/2$ Theorem~\ref{thm:szwarcturan} (i) can be applied and for $\alpha\leq-1/2$ Theorem~\ref{thm:szwarcturan} (ii) is appropriate \cite[Section 5]{Sz98}. Further general criteria can be found in \cite{BS09,Kr05,Kr11,Ro78,Ro80,Sz98,Sz21}. The importance of general criteria in the context of Tur\'{a}n's inequality was particularly pointed out by Nevai in \cite[Problem 2.1]{vA10}.\\

Motivated by Gasper's result on Jacobi polynomials, in this paper we study a closely related generalization of the ultraspherical polynomials: let $\alpha,\beta>-1$. The generalized Chebyshev polynomials $(T_n^{(\alpha,\beta)}(x))_{n\in\mathbb{N}_0}$ are given via
\begin{equation}\label{eq:genchebrec}
c_{2n-1}=\frac{n+\beta}{2n+\alpha+\beta},\;c_{2n}=\frac{n}{2n+\alpha+\beta+1}
\end{equation}
for $n\in\mathbb{N}$, are orthogonal w.r.t.
\begin{equation}\label{eq:measuregencheb}
(1-x^2)^{\alpha}|x|^{2\beta+1}\chi_{(-1,1)}(x)\,\mathrm{d}x
\end{equation}
and relate to the Jacobi polynomials via the quadratic transformations
\begin{equation}\label{eq:quadratic}
T_{2n}^{(\alpha,\beta)}(x)=R_n^{(\alpha,\beta)}(2x^2-1),\;T_{2n+1}^{(\alpha,\beta)}(x)=x R_n^{(\alpha,\beta+1)}(2x^2-1)
\end{equation}
for $n\in\mathbb{N}_0$ \cite[Chapter V 2 (G)]{Ch78} \cite[Section 3 (f)]{La83}. Hypergeometric representations are available \cite{KLS10,MK12}. The ultraspherical polynomials are contained as special case $\beta=-1/2$. Concerning another generalization of the ultraspherical polynomials, the associated symmetric Pollaczek polynomials, Tur\'{a}n's inequality was recently established via Theorem~\ref{thm:szwarcturan} \cite{Ka21b}. Also the continuous $q$-ultraspherical polynomials can be tackled with Theorem~\ref{thm:szwarcturan} \cite{Sz98}. However, it is easy to see that the conditions of Theorem~\ref{thm:szwarcturan} are not fulfilled for $(T_n^{(\alpha,\beta)}(x))_{n\in\mathbb{N}_0}$ besides the special case $\beta=-1/2$. One of the central results of the present paper will be that $(T_n^{(\alpha,\beta)}(x))_{n\in\mathbb{N}_0}$ satisfies Tur\'{a}n's inequality if and only if $\beta\leq0$ (where the nontrivial part is to show that $\beta\leq0$ is sufficient), and we will obtain this by establishing new general criteria in the spirit of Nevai's \cite[Problem 2.1]{vA10}. These new general criteria rely on the recurrence relation \eqref{eq:threetermrec}, too, and one of them generalizes Theorem~\ref{thm:szwarcturan}. The paper is organized as follows: in Section~\ref{sec:general}, we formulate and prove the announced general criteria. We also deal with estimations of the form $\Delta_n(x)\geq K_n\cdot(1-x^2)\;(x\in[-1,1])$ with $K_n>0$ (which clearly imply Tur\'{a}n's inequality).\footnote{Such estimations were also studied in \cite{BS09,Ni23}, for instance.} In Section~\ref{sec:application}, we present the application of these general criteria to generalized Chebyshev polynomials and obtain the announced result on this class. We also obtain two very explicit nonnegative representations of the Tur\'{a}n determinants for every $\beta\leq0$.\footnote{With ``nonnegative representations'' of the Tur\'{a}n determinants we mean representations which directly show their nonnegativity on $[-1,1]$.} Furthermore, we briefly compare our approach to Gasper's above-mentioned proof for Jacobi polynomials and roughly sketch another variant for the class of generalized Chebyshev polynomials in this context.\footnote{We think that the presentation of several variants in the context of Tur\'{a}n's inequality is in good tradition: for instance, in \cite{Sz48} Szeg{\H{o}} gave four proofs for the Legendre polynomials, and the different proofs shed light on various aspects of the polynomials under consideration, are more or less suitable for generalizations and so on.} Moreover, in Section~\ref{sec:application} we present some further observations, including a more detailed study of the polynomials $(T_n^{(\alpha,0)}(x))_{n\in\mathbb{N}_0}$ and a comparison to the Jacobi polynomials. In Section~\ref{sec:2sieved}, we study certain sieved polynomials and briefly discuss the special case of $2$-sieved ultraspherical polynomials, which corresponds to the polynomials $(T_n^{(\alpha,\alpha)}(x))_{n\in\mathbb{N}_0}$. Finally, in Section~\ref{sec:furtherexamples} we turn back to our general criteria again and discuss further examples and ideas.\\

We mention that we used computer algebra systems (Maple) to find suitable simplifications and decompositions, find explicit formulas and estimations, get conjectures and so on. The final proofs can be understood without any computer usage, however.

\section{General criteria for Tur\'{a}n's inequality}\label{sec:general}

During the whole section, let $(P_n(x))_{n\in\mathbb{N}_0}$ be as in the general setting described in Section~\ref{sec:intro}. We first observe that $\Delta_1(x)=K_1\cdot(1-x^2)$ with $K_1:=c_1/a_1>0$. Moreover, one has $\Delta_2(x)=((a_1-a_2)x^2+c_1^2a_2)/(a_1^2a_2)\cdot(1-x^2)$ and $\Delta_2(x)=((a_2-a_1)(1-x^2)+a_1((1+c_1)c_2-c_1))/(a_1^2a_2)\cdot(1-x^2)$. Therefore, $\Delta_2(x)>0\;(x\in(-1,1))$ if and only if $c_2\geq c_1/(1+c_1)$, and if $c_2<c_1/(1+c_1)$, there even exists a constant $K_2>0$ such that $\Delta_2(x)\geq K_2\cdot(1-x^2)$ for all $x\in[-1,1]$. If $c_2<c_1/(1+c_1)$, however, then $\Delta_2(x)<0$ if $x\in(-1,1)$ is sufficiently close to $1$.\\

The following result is the first main result of this paper:

\begin{theorem}\label{thm:kahlerturan}
Let $c_2\geq c_1/(1+c_1)$, and let one of the two alternatives $0\leq A_n\leq B_n\leq C_n$, $0\geq A_n\geq B_n\geq C_n$ be satisfied for every $n\in\mathbb{N}$, where $A_n:=c_n(a_{n+2}-c_{n+2})$, $B_n:=(a_n-c_{n+2})c_{n+1}$ and $C_n:=(a_n-c_n)c_{n+2}$. Then the polynomials $(P_n(x))_{n\in\mathbb{N}_0}$ satisfy Tur\'{a}n's inequality. Moreover, if additionally $c_2>c_1/(1+c_1)$, then one has the following sharper estimation: for every $n\in\mathbb{N}$, there exists a constant $K_n>0$ such that $\Delta_n(x)\geq K_n\cdot(1-x^2)$ for all $x\in[-1,1]$.
\end{theorem}

Note that Theorem~\ref{thm:kahlerturan} is a generalization of Theorem~\ref{thm:szwarcturan}, which can be seen as follows: under the conditions of Theorem~\ref{thm:szwarcturan} (i), one has $c_2\geq c_1>c_1/(1+c_1)$; moreover, for every $n\in\mathbb{N}$ the first of the two alternatives in Theorem~\ref{thm:kahlerturan} is obviously satisfied. Under the conditions of Theorem~\ref{thm:szwarcturan} (ii), however, one can estimate $c_2\geq1/2>c_1/(1+c_1)$, and for every $n\in\mathbb{N}$ the second of the two alternatives in Theorem~\ref{thm:kahlerturan} is satisfied. The latter can be seen as follows: under the conditions of Theorem~\ref{thm:szwarcturan} (ii), one has $(c_{n+2}-a_n)c_{n+2}-c_n(c_{n+2}-a_{n+2})=(c_n-c_{n+2})a_{n+2}\geq0$ and $(c_n-a_n)c_{n+2}-(c_{n+2}-a_n)c_n=(c_n-c_{n+2})a_n\geq0$ for every $n\in\mathbb{N}$. In particular, the additional condition ``$c_2>c_1/(1+c_1)$'' concerning the second assertion of Theorem~\ref{thm:kahlerturan} is necessarily fulfilled under the conditions of Theorem~\ref{thm:szwarcturan}.\\

We now come to the proof of Theorem~\ref{thm:kahlerturan}. It will be done by induction, and the crucial idea compared to Szwarc's result Theorem~\ref{thm:szwarcturan} is to study implications of the form ``$\Delta_n(x)\geq0$ $\Rightarrow$ $\Delta_{n+2}(x)\geq0$'' rather than implications of the form ``$\Delta_n(x)\geq0$ $\Rightarrow$ $\Delta_{n+1}(x)\geq0$''.

\begin{proof}[Proof of Theorem~\ref{thm:kahlerturan}]
Let $n\in\mathbb{N}$. By the recurrence relation \eqref{eq:threetermrec}, we have
\begin{equation}\label{eq:Deltanexpand}
\Delta_n(x)=\frac{a_n}{c_n}P_{n+1}^2(x)-\frac{x}{c_n}P_{n+1}(x)P_n(x)+P_n^2(x)
\end{equation}
(such an expansion can be motivated from Tur\'{a}n's paper \cite{Tu50}) and
\begin{equation}\label{eq:Deltanplus2expand}
\begin{split}
\Delta_{n+2}(x)&=\frac{(a_{n+2}-a_{n+1})x^2+a_{n+1}^2c_{n+2}}{a_{n+1}^2a_{n+2}}P_{n+1}^2(x)+\frac{(a_{n+1}-2a_{n+2})c_{n+1}x}{a_{n+1}^2a_{n+2}}P_{n+1}(x)P_n(x)\\
&\quad+\frac{c_{n+1}^2}{a_{n+1}^2}P_n^2(x).
\end{split}
\end{equation}
Combining \eqref{eq:Deltanexpand} and \eqref{eq:Deltanplus2expand}, we obtain the essential ingredient
\begin{equation}\label{eq:DeltanDeltanplus2combined}
\begin{split}
&a_{n+1}^2a_{n+2}C_n\Delta_{n+2}(x)-a_{n+1}c_{n+1}c_{n+2}A_n\Delta_n(x)\\
&=a_{n+1}c_{n+2}(C_n-B_n)(1-x^2)P_{n+1}^2(x)+c_{n+1}c_{n+2}(B_n-A_n)(x P_{n+1}(x)-P_n(x))^2
\end{split}
\end{equation}
after a somewhat tedious calculation.\footnote{We particularly emphasize that computer algebra systems (Maple) were very helpful in finding the ``right'' combination of $\Delta_{n+2}(x)$ and $\Delta_n(x)$ in \eqref{eq:DeltanDeltanplus2combined} (cf. also the general remark at the end of Section~\ref{sec:intro}).} Now we can conclude as follows:
\begin{enumerate}[(i)]
\item If $A_n\neq0$, then \eqref{eq:DeltanDeltanplus2combined} implies that $\Delta_{n+2}(x)\geq c_{n+1}c_{n+2}A_n/(a_{n+1}a_{n+2}C_n)\cdot\Delta_n(x)$ for all $x\in[-1,1]$, where either $A_n$ and $C_n$ are both positive or $A_n$ and $C_n$ are both negative.
\item If $A_n=0$, then $a_{n+2}=c_{n+2}=1/2$. Consequently, \eqref{eq:Deltanexpand} (with $n$ replaced by $n+2$) yields $\Delta_{n+2}(x)=(P_{n+3}(x)-x P_{n+2}(x))^2+(1-x^2)P_{n+2}^2(x)$ and $\Delta_{n+2}(x)=(P_{n+2}(x)-x P_{n+3}(x))^2+(1-x^2)P_{n+3}^2(x)$. Therefore, $\Delta_{n+2}(x)\geq K_n\cdot(1-x^2)$ for all $x\in[-1,1]$, where $K_n:=\min_{x\in[-1,1]}\max\{P_{n+2}^2(x),P_{n+3}^2(x)\}$ is positive because otherwise $P_{n+2}(x)$ and $P_{n+3}(x)$ would have a common zero (a related argument can be found in Tur\'{a}n's paper \cite{Tu50}).
\end{enumerate}
Therefore, for every $n\in\mathbb{N}$ we have shown the implication ``$\Delta_n(x)>0\;(x\in(-1,1))$ $\Rightarrow$ $\Delta_{n+2}(x)>0\;(x\in(-1,1))$'', as well as the implication ``($\Delta_n(x)\geq K_n\cdot(1-x^2)\;(x\in[-1,1])$ for some $K_n>0$) $\Rightarrow$ ($\Delta_{n+2}(x)\geq K_{n+2}\cdot(1-x^2)\;(x\in[-1,1])$ for some $K_{n+2}>0$). Concerning $\Delta_1(x)$ and $\Delta_2(x)$, we refer to the beginning of the section.
\end{proof}

Before coming to our second general criterion for Tur\'{a}n's inequality, we need some preliminaries and additional notation. For each $m\in\mathbb{N}_0$, let $(P_{m,n}(x))_{n\in\mathbb{N}_0}$ be orthogonal w.r.t. $(1-x^2)^m\,\mathrm{d}\mu(x)$ and normalized by $P_{m,n}(1)\equiv1$. The polynomials $(P_{m,n}(x))_{n\in\mathbb{N}_0}$ relate to $(P_{m+1,n}(x))_{n\in\mathbb{N}_0}$ via
\begin{equation}\label{eq:Pnast}
P_{m+1,n}(x)=C_{m,n}\frac{P_{m,n+2}(x)-P_{m,n}(x)}{1-x^2},
\end{equation}
where $C_{m,n}<0$ is independent of $x$ \cite{BS09} (the reference only considers the special case $m=0$).\footnote{Formulas for $C_{0,n}$ can be found in \cite[Section 1]{Ka23b}.} Moreover, we set $c_{m,0}\equiv0$ and define $(c_{m,n})_{n\in\mathbb{N}}$ and $(a_{m,n})_{n\in\mathbb{N}}$ by the expansions
\begin{equation}\label{eq:threetermrecast}
x P_{m,n}(x)=a_{m,n}P_{m,n+1}(x)+c_{m,n}P_{m,n-1}(x)\;(n\in\mathbb{N}_0).
\end{equation}
It is clear that $a_{m,n}\equiv1-c_{m,n}$, and we have
\begin{equation}\label{eq:astreccaststart}
C_{m,0}=-a_{m,1}
\end{equation}
and
\begin{equation}\label{eq:astreccast}
\frac{C_{m,n}}{C_{m,n-1}}=\frac{a_{m,n+1}}{a_{m+1,n-1}}=\frac{c_{m+1,n}}{c_{m,n}}\;(n\in\mathbb{N}).
\end{equation}
\eqref{eq:astreccaststart} is immediate from \eqref{eq:Pnast} and \eqref{eq:threetermrecast}. \eqref{eq:astreccast} can be seen as follows: multiplying \eqref{eq:threetermrecast} (with $m$ replaced by $m+1$) by $1-x^2$ and combining the result with \eqref{eq:Pnast} and \eqref{eq:threetermrecast} (now for $m$), we get
\begin{align*}
&a_{m,n+2}C_{m,n}P_{m,n+3}(x)+(c_{m,n+2}-a_{m,n})C_{m,n}P_{m,n+1}(x)-c_{m,n}C_{m,n}P_{m,n-1}(x)\\
&=a_{m+1,n}C_{m,n+1}P_{m,n+3}(x)+(c_{m+1,n}C_{m,n-1}-a_{m+1,n}C_{m,n+1})P_{m,n+1}(x)\\
&\quad-c_{m+1,n}C_{m,n-1}P_{m,n-1}(x)
\end{align*}
for all $n\in\mathbb{N}_0$. Comparing the leading coefficients of $P_{m,n+3}(x)$ and $P_{m,n-1}(x)$, we obtain \eqref{eq:astreccast}. \eqref{eq:astreccast} particularly yields that
\begin{equation}\label{eq:astrecm}
a_{m,n+1}c_{m,n}\equiv c_{m+1,n}a_{m+1,n-1}.
\end{equation}
In the language of chain sequences \cite{Ch78,Wa48}, \eqref{eq:astrecm} can be expressed as follows: $(c_{m+1,n})_{n\in\mathbb{N}_0}$ is the minimal parameter sequence of $(a_{m,n+1}c_{m,n})_{n\in\mathbb{N}}$. In the following, let $\Delta_{m,n}(x):=(P_{m,n}(x))^2-P_{m,n+1}(x)P_{m,n-1}(x)$ for every $m\in\mathbb{N}_0$. \cite[Theorem 2.3]{BS09} and \eqref{eq:astreccast} yield:

\begin{lemma}\label{lma:bergszwarcturan}
One has $\Delta_{n+1}(x)=s_n\cdot(1-x^2)P_{1,n}^2(x)+t_n\cdot(1-x^2)\Delta_{1,n}(x)$ for all $n\in\mathbb{N}_0$, where $s_n:=(a_{n+1}c_{n+1}-a_{1,n}c_{1,n})/C_{0,n}^2$, $t_n:=a_{1,n}c_{1,n}/C_{0,n}^2$.
\end{lemma}

The following result is based on Lemma~\ref{lma:bergszwarcturan} and the second main result of this paper:

\begin{theorem}\label{thm:kahlerturan2}
Let $a_{m,n+1}c_{m,n+1}\geq a_{m+1,n}c_{m+1,n}$ for all $m\in\mathbb{N}_0$ and $n\in\mathbb{N}$. Then the polynomials $(P_n(x))_{n\in\mathbb{N}_0}$ satisfy Tur\'{a}n's inequality, and one has the nonnegative representation
\begin{equation}\label{eq:Deltamnexplred}
\Delta_n(x)=\sum_{k=1}^n(1-x^2)^k P_{k,n-k}^2(x)s_{k-1,n-k}\prod_{j=1}^{k-1}t_{j-1,n-j}\;(n\in\mathbb{N})
\end{equation}
of the Tur\'{a}n determinants, where $s_{m,n}:=(a_{m,n+1}c_{m,n+1}-a_{m+1,n}c_{m+1,n})/C_{m,n}^2$ and $t_{m,n}:=a_{m+1,n}c_{m+1,n}/C_{m,n}^2$ for $m,n\in\mathbb{N}_0$. Moreover, if additionally $a_{n+1}c_{n+1}>a_{1,n}c_{1,n}$ for all $n\in\mathbb{N}$, then for every $n\in\mathbb{N}$ there exists a constant $K_n>0$ such that $\Delta_n(x)\geq K_n\cdot(1-x^2)$ for all $x\in[-1,1]$.
\end{theorem}

\begin{proof}
We show an (only apparently) stronger assertion than \eqref{eq:Deltamnexplred}, namely that, for every $n\in\mathbb{N}$,
\begin{equation}\label{eq:Deltamnexpl}
\Delta_{m,n}(x)=\sum_{k=1}^n(1-x^2)^k P_{m+k,n-k}^2(x)s_{m+k-1,n-k}\prod_{j=1}^{k-1}t_{m+j-1,n-j}\;(m\in\mathbb{N}_0).
\end{equation}
For $n=1$, due to \eqref{eq:astreccaststart} both sides of \eqref{eq:Deltamnexpl} reduce to $c_{m,1}/a_{m,1}\cdot(1-x^2)$. Now let $n\in\mathbb{N}$ be arbitrary but fixed, and assume that \eqref{eq:Deltamnexpl} is true for $n$. Then, as a consequence of Lemma~\ref{lma:bergszwarcturan}, we have
\begin{align*}
&\Delta_{m,n+1}(x)-s_{m,n}\cdot(1-x^2)P_{m+1,n}^2(x)\\
&=t_{m,n}\cdot(1-x^2)\Delta_{m+1,n}(x)\\
&=t_{m,n}\cdot(1-x^2)\sum_{k=1}^n(1-x^2)^k P_{m+1+k,n-k}^2(x)s_{m+k,n-k}\prod_{j=1}^{k-1}t_{m+j,n-j}\\
&=\sum_{k=1}^n(1-x^2)^{k+1}P_{m+1+k,n-k}^2(x)s_{m+k,n-k}\prod_{j=0}^{k-1}t_{m+j,n-j}\\
&=\sum_{k=1}^n(1-x^2)^{k+1}P_{m+1+k,n-k}^2(x)s_{m+k,n-k}\prod_{j=1}^k t_{m+j-1,n+1-j}\\
&=\sum_{k=2}^{n+1}(1-x^2)^k P_{m+k,n+1-k}^2(x)s_{m+k-1,n+1-k}\prod_{j=1}^{k-1}t_{m+j-1,n+1-j}
\end{align*}
for all $m\in\mathbb{N}_0$. Hence, \eqref{eq:Deltamnexpl} is satisfied for $n+1$, too. Considering the special case $m=0$ of \eqref{eq:Deltamnexpl}, we obtain the representation \eqref{eq:Deltamnexplred}. By the condition, we have $s_{m,n}\geq0$ for all $m,n\in\mathbb{N}_0$ and $s_{m,0}>0$ for all $m\in\mathbb{N}_0$. Since $t_{m,n}>0$ for all $m\in\mathbb{N}_0$ and $n\in\mathbb{N}$, \eqref{eq:Deltamnexplred} establishes Tur\'{a}n's inequality. We now additionally assume that $a_{n+1}c_{n+1}>a_{1,n}c_{1,n}$ (and therefore $s_{0,n}>0$) for all $n\in\mathbb{N}$. Then for every $n\geq2$ \eqref{eq:Deltamnexplred} implies that
\begin{equation*}
\Delta_n(x)\geq\left(P_{1,n-1}^2(x)s_{0,n-1}+(1-x^2)^{n-1}s_{n-1,0}\prod_{j=1}^{n-1}t_{j-1,n-j}\right)\cdot(1-x^2)\geq K_n\cdot(1-x^2)
\end{equation*}
for all $x\in[-1,1]$, where $K_n>0$ can be chosen as the minimum of $P_{1,n-1}^2(x)s_{0,n-1}+(1-x^2)^{n-1}s_{n-1,0}\prod_{j=1}^{n-1}t_{j-1,n-j}$ on $[-1,1]$. Recall that the case $n=1$ is clear.
\end{proof}

\begin{corollary}\label{cor:kahlerturan2}
Let $c_{m+1,n}\leq c_{m,n+1}$ for all $m\in\mathbb{N}_0$ and $n\in\mathbb{N}$. Then the polynomials $(P_n(x))_{n\in\mathbb{N}_0}$ satisfy Tur\'{a}n's inequality with nonnegative representation \eqref{eq:Deltamnexplred}. Moreover, if additionally $c_{1,n}<c_{n+1}$ for all $n\in\mathbb{N}$, then for every $n\in\mathbb{N}$ there exists a constant $K_n>0$ such that $\Delta_n(x)\geq K_n\cdot(1-x^2)$ for all $x\in[-1,1]$.
\end{corollary}

\begin{proof}
We first \textit{claim} that if $(c_n)_{n\in\mathbb{N}}\subseteq(0,1)$ is arbitrary, then $c_{m+1,n}<a_{m,n+1}$ for all $m\in\mathbb{N}_0$ and $n\in\mathbb{N}$. This is a consequence of \eqref{eq:astrecm}: indeed, we have $c_{m+1,1}=a_{m,2}c_{m,1}<a_{m,2}$. Furthermore, if $n\in\mathbb{N}$ is arbitrary but fixed and $c_{m+1,n}<a_{m,n+1}$, then $a_{m,n+2}c_{m,n+1}=c_{m+1,n+1}a_{m+1,n}>c_{m+1,n+1}c_{m,n+1}$, so $c_{m+1,n+1}<a_{m,n+2}$. This establishes the claim. Therefore, by the condition we have $c_{m+1,n}\leq\min\{c_{m,n+1},a_{m,n+1}\}\leq1/2$ for all $m\in\mathbb{N}_0$ and $n\in\mathbb{N}$, and the first assertion follows from Theorem~\ref{thm:kahlerturan2}. If additionally $c_{1,n}<c_{n+1}$ for all $n\in\mathbb{N}$, then $c_{1,n}<\min\{c_{n+1},a_{n+1}\}\leq1/2$ for all $n\in\mathbb{N}$, so the second assertion follows from Theorem~\ref{thm:kahlerturan2}, too.
\end{proof}

\begin{remark}
Recall that our first general criterion Theorem~\ref{thm:kahlerturan} generalizes Theorem~\ref{thm:szwarcturan}, so the reader might ask whether Theorem~\ref{thm:kahlerturan2} is also a generalization of Theorem~\ref{thm:szwarcturan}. The following observations show that this is not the case: it can happen that $(c_n)_{n\in\mathbb{N}}\subseteq(0,1/2]$ is nondecreasing but $1/2>c_{m+1,n}>c_{m,n+1}$ for some $m,n\in\mathbb{N}$. An explicit example is provided by $c_1:=c_2:=\frac{1}{4}$ and $c_n:=\frac{1}{2}$ for $n\geq3$, where $c_{2,1}=33/208>2/13=c_{1,2}$. Furthermore, it can happen that $(c_n)_{n\in\mathbb{N}}\subseteq[1/2,1)$ is nonincreasing but $1/2>c_{m+1,n}>c_{m,n+1}$ for some $m,n\in\mathbb{N}$. An explicit example is provided by $c_1:=c_2:=c_3:=\frac{4}{5}$ and $c_n:=\frac{1}{2}$ for $n\geq4$, where $c_{3,1}=1316/11425>860/7769=c_{2,2}$.
\end{remark}

\begin{remark}
As analogue to \eqref{eq:Deltamnexplred}, a nonnegative representation of the Tur\'{a}n determinants in the spirit of Theorem~\ref{thm:kahlerturan} is available as a consequence of \eqref{eq:DeltanDeltanplus2combined}. However, it generally requires case differentiations (cf. the proof of Theorem~\ref{thm:kahlerturan}). In Section~\ref{sec:application}, we study a specific class of polynomials and can avoid such case differentiations.
\end{remark}

\section{Application to the class of generalized Chebyshev polynomials}\label{sec:application}

We now come to the announced result concerning Tur\'{a}n's inequality for the generalized Chebyshev polynomials, which is the third main result of our paper, and we will obtain this both as a consequence of Theorem~\ref{thm:kahlerturan} and as a consequence of Theorem~\ref{thm:kahlerturan2}/Corollary~\ref{cor:kahlerturan2}. As a preparation, we mention the following concerning the constants $C_{m,n}$ (cf. \eqref{eq:Pnast}): if $(P_n(x))_{n\in\mathbb{N}_0}=(T_n^{(\alpha,\beta)}(x))_{n\in\mathbb{N}_0}$ (with $\alpha,\beta>-1$), then
\begin{equation}\label{eq:Cnastgencheb}
C_{m,n}=-\frac{m+\alpha+1}{m+n+\alpha+\beta+2}\;(m,n\in\mathbb{N}_0).
\end{equation}
This can easily be seen from \eqref{eq:Pnast} by comparing the leading coefficients of $P_{m+1,n}(x)=T_n^{(m+\alpha+1,\beta)}(x)$ and $P_{m,n+2}(x)=T_{n+2}^{(m+\alpha,\beta)}(x)$ (which can be computed from \eqref{eq:threetermrec} and \eqref{eq:genchebrec}). We also recall the definition of the Pochhammer symbol $(a)_n:=\prod_{k=1}^n(a+k-1)$.

\begin{theorem}\label{thm:genchebTuran}
Let $(P_n(x))_{n\in\mathbb{N}_0}=(T_n^{(\alpha,\beta)}(x))_{n\in\mathbb{N}_0}$ with $\alpha,\beta>-1$. Then Tur\'{a}n's inequality is satisfied if and only if $\beta\leq0$. Moreover, if $\beta<0$, then for every $n\in\mathbb{N}$ there exists a constant $K_n>0$ such that $\Delta_n(x)\geq K_n\cdot(1-x^2)$ for all $x\in[-1,1]$. Furthermore, one has the following nonnegative representations of the Tur\'{a}n determinants (provided $\beta\in(-1,0]$): for every $n\in\mathbb{N}$, one has
\begin{align*}
&\Delta_{2n-1}(x)\\
&=\frac{(\beta+1)(n-1)!(\beta+1)_{n-1}}{(\alpha+1)_n(\alpha+\beta+2)_{n-1}}(1-x^2)\\
&\quad+\sum_{k=1}^{n-1}\frac{(2k+\alpha+\beta+1)(k+\beta+1)_{n-1-k}(k+1)_{n-1-k}}{(k+\alpha+\beta+1)(k+\alpha+1)_{n-k}(k+\alpha+\beta+1)_{n-k}}\\
&\quad\quad\quad\quad\times\left[(\beta+1)(k+\alpha+\beta+1)\left(T_{2k}^{(\alpha,\beta)}(x)\right)^2(1-x^2)-\beta k\left(x T_{2k}^{(\alpha,\beta)}(x)-T_{2k-1}^{(\alpha,\beta)}(x)\right)^2\right]\\
&=\frac{\beta+1}{\alpha+1}\left(T_{2n-2}^{(\alpha+1,\beta)}(x)\right)^2(1-x^2)\\
&\quad+\sum_{k=1}^{n-1}\frac{(n+\alpha+\beta+1)_{n-k}(n+\alpha+1)_{n-1-k}(k+\beta+1)_{n-1-k}(k)_{n-k}}{(2n-2k+\alpha+1)(\alpha+1)_{2n-2k}^2}\\
&\quad\quad\quad\quad\times\left[(\beta+1)(2n-k+\alpha)(k+\beta)\left(T_{2k-2}^{(2n-2k+\alpha+1,\beta)}(x)\right)^2(1-x^2)\right.\\
&\quad\quad\quad\quad\quad\quad\left.-\beta(2n-2k+\alpha+1)(2n-2k+\alpha)\left(T_{2k-1}^{(2n-2k+\alpha,\beta)}(x)\right)^2\right](1-x^2)^{2n-2k}
\end{align*}
and
\begin{align*}
&\Delta_{2n}(x)\\
&=\sum_{k=0}^{n-1}\frac{(2k+\alpha+\beta+2)(k+\beta+2)_{n-1-k}(k+1)_{n-1-k}}{(k+\alpha+1)(k+\alpha+1)_{n-k}(k+\alpha+\beta+2)_{n-k}}\\
&\quad\quad\quad\times\left[-\beta(k+\alpha+1)\left(T_{2k+1}^{(\alpha,\beta)}(x)\right)^2(1-x^2)\right.\\
&\quad\quad\quad\quad\quad\left.+(\beta+1)(k+\beta+1)\left(x T_{2k+1}^{(\alpha,\beta)}(x)-T_{2k}^{(\alpha,\beta)}(x)\right)^2\right]\\
&=\sum_{k=0}^{n-1}\frac{(n+\alpha+\beta+2)_{n-1-k}(n+\alpha+1)_{n-1-k}(k+\beta+2)_{n-1-k}(k+1)_{n-1-k}}{(2n-2k+\alpha)(\alpha+1)_{2n-1-2k}^2}\\
&\quad\quad\quad\times\left[-\beta(2n-2k+\alpha)(2n-2k+\alpha-1)\left(T_{2k+1}^{(2n-2k+\alpha-1,\beta)}(x)\right)^2\right.\\
&\quad\quad\quad\quad\quad\left.+(\beta+1)(2n-k+\alpha)(k+\beta+1)\left(T_{2k}^{(2n-2k+\alpha,\beta)}(x)\right)^2(1-x^2)\right]\\
&\quad\quad\quad\times(1-x^2)^{2n-1-2k}.
\end{align*}
\end{theorem}

For the special case $\beta=-1/2$ (ultraspherical polynomials), the existence of the constants $K_n$ (cf. above) was already observed in \cite{TN51} and also considered in \cite{BS09,Ni23,VL57}.

\begin{proof}[Proof of Theorem~\ref{thm:genchebTuran}]
We have $c_2-c_1/(1+c_1)=-\beta(\alpha+\beta+2)/((\alpha+\beta+3)(\alpha+2\beta+3))$, which particularly shows the necessity of $\beta\leq0$ for Tur\'{a}n's inequality by a consideration of $\Delta_2(x)$ (cf. the beginning of Section~\ref{sec:general}). The nontrivial part is the sufficiency. We use the notation of Theorem~\ref{thm:kahlerturan} and Theorem~\ref{thm:kahlerturan2}. From \eqref{eq:genchebrec}, for every $n\in\mathbb{N}$ we compute $A_{2n-1}=(\alpha-\beta)(n+\beta)/((2n+\alpha+\beta)(2n+\alpha+\beta+2))$, $A_{2n}=(\alpha+\beta+1)n/((2n+\alpha+\beta+1)(2n+\alpha+\beta+3))$, $B_{2n-1}=(\alpha-\beta)n/((2n+\alpha+\beta)(2n+\alpha+\beta+2))$, $B_{2n}=(\alpha+\beta+1)(n+\beta+1)/((2n+\alpha+\beta+1)(2n+\alpha+\beta+3))$, $C_{2n-1}=(\alpha-\beta)(n+\beta+1)/((2n+\alpha+\beta)(2n+\alpha+\beta+2))$, $C_{2n}=(\alpha+\beta+1)(n+1)/((2n+\alpha+\beta+1)(2n+\alpha+\beta+3))$. Now let $m\in\mathbb{N}_0$. Since, due to \eqref{eq:measuregencheb}, a transition from $(P_n(x))_{n\in\mathbb{N}_0}$ to $(P_{m,n}(x))_{n\in\mathbb{N}_0}$ corresponds to a transition from $\alpha$ to $m+\alpha$, we have $c_{m,2n-1}=(n+\beta)/(2n+m+\alpha+\beta)$, $c_{m,2n}=n/(2n+m+\alpha+\beta+1)$ as a consequence of \eqref{eq:genchebrec}.\footnote{Alternatively, these formulas for $c_{m,2n-1}$ and $c_{m,2n}$ follow from \eqref{eq:genchebrec} and the general recurrence relation \eqref{eq:astrecm}.} Hence, we have $c_{m,2n}-c_{m+1,2n-1}=-\beta/(2n+m+\alpha+\beta+1)$ and $c_{m,2n+1}-c_{m+1,2n}=(\beta+1)/(2n+m+\alpha+\beta+2)$ for all $n\in\mathbb{N}$. Therefore, the first assertions follow immediately from Theorem~\ref{thm:kahlerturan}, as well as from Theorem~\ref{thm:kahlerturan2}/Corollary~\ref{cor:kahlerturan2}. It remains to establish the nonnegative representations. These can be seen as follows: as a consequence of \eqref{eq:DeltanDeltanplus2combined}, we have
\begin{align*}
\Delta_{2n+1}(x)&=\frac{n(n+\beta)}{(n+\alpha+1)(n+\alpha+\beta+1)}\Delta_{2n-1}(x)\\
&\quad+\frac{(\beta+1)(2n+\alpha+\beta+1)}{(n+\alpha+1)(n+\alpha+\beta+1)}(1-x^2)P_{2n}^2(x)\\
&\quad+\frac{-\beta n(2n+\alpha+\beta+1)}{(n+\alpha+1)(n+\alpha+\beta+1)^2}(x P_{2n}(x)-P_{2n-1}(x))^2
\end{align*}
and
\begin{align*}
\Delta_{2n+2}(x)&=\frac{n(n+\beta+1)}{(n+\alpha+1)(n+\alpha+\beta+2)}\Delta_{2n}(x)\\
&\quad+\frac{-\beta(2n+\alpha+\beta+2)}{(n+\alpha+1)(n+\alpha+\beta+2)}(1-x^2)P_{2n+1}^2(x)\\
&\quad+\frac{(\beta+1)(n+\beta+1)(2n+\alpha+\beta+2)}{(n+\alpha+1)^2(n+\alpha+\beta+2)}(x P_{2n+1}(x)-P_{2n}(x))^2
\end{align*}
for all $n\in\mathbb{N}$. These recurrence relations and the formulas for $\Delta_1(x)$ and $\Delta_2(x)$ at the beginning of Section~\ref{sec:general} yield the first representations of $\Delta_{2n-1}(x)$ and $\Delta_{2n}(x)$. Moreover, \eqref{eq:Deltamnexplred} and \eqref{eq:Cnastgencheb} yield $s_{m,2n}=(\beta+1)/(m+\alpha+1)$, $s_{m,2n+1}=-\beta/(m+\alpha+1)$, $t_{m,2n}=(m+n+\alpha+\beta+2)n/(m+\alpha+1)^2$ and $t_{m,2n+1}=(m+n+\alpha+2)(n+\beta+1)/(m+\alpha+1)^2$ for all $m,n\in\mathbb{N}_0$ and the second representations.
\end{proof}

We now roughly sketch a variant for the proof of Tur\'{a}n's inequality for generalized Chebyshev polynomials which is more based on their specific structure and partially in the spirit of Gasper's proof for Jacobi polynomials. More precisely, we follow ideas from \cite{Ga72,Sk54,TN51} for ultraspherical and Jacobi polynomials. This variant relies on the relation \eqref{eq:quadratic} to the Jacobi polynomials, on differential equations and on the Jacobi forward shift. Moreover, at an appropriate stage it will make use of $(P_{1,n}(x))_{n\in\mathbb{N}_0}$ again but avoid considerations of $(P_{m,n}(x))_{n\in\mathbb{N}_0}$ for $m\geq2$. Let $(P_n(x))_{n\in\mathbb{N}_0}=(T_n^{(\alpha,\beta)}(x))_{n\in\mathbb{N}_0}$ with $\alpha>-1$ and $\beta\in(-1,0]$, and let $n\in\mathbb{N}$. Furthermore, let $x_1<\ldots<x_{2n}\in(-1,1)$ denote the zeros of $P_{2n}(x)$ (due to symmetry, one has $x_k+x_{2n+1-k}\equiv0$). As a consequence of the ``even part'' of \eqref{eq:quadratic} and the well-known differential equation for Jacobi polynomials \cite[(9.8.6)]{KLS10}, we obtain the differential equation $x(1-x^2)P_{2n}^{\prime\prime}(x)+(2\beta+1-(2\alpha+2\beta+3)x^2)P_{2n}^{\prime}(x)+4n(n+\alpha+\beta+1)x P_{2n}(x)=0$; moreover, \eqref{eq:genchebrec}, \eqref{eq:quadratic}, \eqref{eq:Cnastgencheb} and the forward shift for Jacobi polynomials \cite[(9.8.7)]{KLS10} yield $(1-x^2)P_{2n}^{\prime}(x)=-2n(n+\alpha+\beta+1)/(2n+\alpha+\beta+1)\cdot(P_{2n+1}(x)-P_{2n-1}(x))$ (cf. also \cite{Be01,Ko61}). Based on these ingredients and the relations $(P_{2n}^{\prime}(x))^2=P_{2n}^{\prime\prime}(x)P_{2n}(x)-(P_{2n}^{\prime}/P_{2n})^{\prime}(x)P_{2n}^2(x)$, $P_{2n}^{\prime}(x)=\sum_{k=1}^{2n}P_{2n}(x)/(x-x_k)$ and $(P_{2n}^{\prime}/P_{2n})^{\prime}(x)P_{2n}^2(x)=-\sum_{k=1}^{2n}P_{2n}^2(x)/(x-x_k)^2$, $\Delta_{2n}(x)$ can be expressed as
\begin{equation}\label{eq:Delta2nfinal}
\Delta_{2n}(x)=\frac{1-x^2}{n(n+\alpha+\beta+1)}\sum_{k=1}^n\frac{(-\beta(1-x_k^2)x^2+(\beta+1)x_k^2(1-x^2))P_{2n}^2(x)}{(x^2-x_k^2)^2}.
\end{equation}
The representation \eqref{eq:Delta2nfinal} shows the nonnegativity of $\Delta_{2n}(x)$ for all $x\in[-1,1]$. The Tur\'{a}n determinant $\Delta_{2n+1}(x)$ requires a different strategy because in general $(1-x^2)P_{2n+1}^{\prime}(x)$ cannot be expressed as a linear combination of $P_{2n+2}(x)$ and $P_{2n}(x)$, so the preceding strategy cannot directly be transferred. This is a major difference compared to Gasper's proof for Jacobi polynomials. However, we can proceed as follows: since $(P_{1,n}(x))_{n\in\mathbb{N}_0}=(T_n^{(\alpha+1,\beta)}(x))_{n\in\mathbb{N}_0}$, we know that also $\Delta_{1,2n}(x)\geq0$ for all $x\in[-1,1]$. Combining \eqref{eq:genchebrec} and \eqref{eq:Cnastgencheb} with Lemma~\ref{lma:bergszwarcturan} (which was obtained from \cite{BS09}), we see that $\Delta_{2n+1}(x)$ is a nonnegative linear combination of $(1-x^2)P_{1,2n}^2(x)$ and $(1-x^2)\Delta_{1,2n}(x)$. Hence, $\Delta_{2n+1}(x)\geq0\;(x\in[-1,1])$. Also the strict inequalities for $x\in(-1,1)$, as well as the estimations of the form $\Delta_n(x)\geq K_n\cdot(1-x^2)$ for $\beta<0$, can be worked out with this approach.\\

We now turn back to our general criteria Theorem~\ref{thm:kahlerturan}, Theorem~\ref{thm:kahlerturan2} and Corollary~\ref{cor:kahlerturan2} and ask the question whether the second assertions remain true if the additional conditions ``$c_2>c_1/(1+c_1)$'', ``$a_{n+1}c_{n+1}>a_{1,n}c_{1,n}$ for all $n\in\mathbb{N}$'' and ``$c_{1,n}<c_{n+1}$ for all $n\in\mathbb{N}$'' are dropped. Moreover, we ask the question whether the second assertion of Theorem~\ref{thm:genchebTuran} remains true if the condition ``$\beta<0$'' is replaced by the weaker condition ``$\beta\leq0$''. Both answers are negative, which follows from the representations given in Theorem~\ref{thm:genchebTuran} and the representation \eqref{eq:Delta2nfinal}:

\begin{corollary}\label{cor:kahlerturanmod}
Let $(P_n(x))_{n\in\mathbb{N}_0}=(T_n^{(\alpha,0)}(x))_{n\in\mathbb{N}_0}$ with $\alpha>-1$. Then $\lim_{x\to1}\Delta_{2n}(x)/(1-x^2)=0$ for every $n\in\mathbb{N}$.
\end{corollary}

\begin{remark}
Corollary~\ref{cor:kahlerturanmod} also demonstrates a remarkable difference between the generalized Chebyshev and the Jacobi polynomials, which are their ``ancestors'' w.r.t. quadratic transformations: if $(P_n(x))_{n\in\mathbb{N}_0}=(R_n^{(\alpha,\beta)}(x))_{n\in\mathbb{N}_0}$ with $\alpha,\beta>-1$, then $\lim_{x\to1}\Delta_n(x)/(1-x^2)=1/(2\alpha+2)\neq0$ for all $n\in\mathbb{N}$. This follows from \cite[p. 438]{Ga72}.
\end{remark}

\section{$2$-sieved polynomials}\label{sec:2sieved}

Let $(P_n(x))_{n\in\mathbb{N}_0}$ be as in the general setting described in Section~\ref{sec:intro} again. The corresponding ``$2$-sieved polynomials'' $(P_n(x;2))_{n\in\mathbb{N}_0}$ are defined via the recurrence relation $P_0(x;2)=1$, $x P_n(x;2)=a(n;2)P_{n+1}(x;2)+c(n;2)P_{n-1}(x;2)\;(n\in\mathbb{N}_0)$ with $c(n;2):=c_{n/2}$ for even $n$ and $c(n;2):=1/2$ for odd $n$, and with $a(n;2)\equiv1-c(n;2)$. Sieved polynomials, as well as various related concepts, have been studied in a series of papers by Ismail et al. In particular, we mention \cite{IL92}. Moreover, we mention the paper \cite{GvA88} by Geronimo and Van Assche. Now let
\begin{equation*}
\Delta_n(x;2):=(P_n(x;2))^2-P_{n+1}(x;2)P_{n-1}(x;2)
\end{equation*}
denote the Tur\'{a}n determinants which belong to $(P_n(x;2))_{n\in\mathbb{N}_0}$. In order to study Tur\'{a}n's inequality for $(P_n(x;2))_{n\in\mathbb{N}_0}$, it is clear that Szwarc's criterion Theorem~\ref{thm:szwarcturan} can only be applied if $c_n\equiv1/2$, i.e., for the Chebyshev polynomials of the first kind. However, as a consequence of Theorem~\ref{thm:kahlerturan} and as the fourth main result of this paper, we obtain the following criterion:

\begin{theorem}\label{thm:2sieved}
\begin{enumerate}[(i)]
\item Let $(c_n)_{n\in\mathbb{N}}\subseteq[1/3,1/2]$, and let $c_{n+1}\geq(1-c_n)/(3-4c_n)$ for all $n\in\mathbb{N}$. Then the polynomials $(P_n(x;2))_{n\in\mathbb{N}_0}$ satisfy Tur\'{a}n's inequality. Moreover, if $c_1>1/3$, then one has the following sharper estimation: for every $n\in\mathbb{N}$, there exists a constant $K_n>0$ such that $\Delta_n(x;2)\geq K_n\cdot(1-x^2)$ for all $x\in[-1,1]$.
\item Let $(c_n)_{n\in\mathbb{N}}\subseteq[1/2,1)$, and let $c_{n+1}\leq(3c_n-1)/(4c_n-1)$ for all $n\in\mathbb{N}$. Then for every $n\in\mathbb{N}$ there exists a constant $K_n>0$ such that $\Delta_n(x;2)\geq K_n\cdot(1-x^2)$ for all $x\in[-1,1]$. In particular, the polynomials $(P_n(x;2))_{n\in\mathbb{N}_0}$ satisfy Tur\'{a}n's inequality.
\end{enumerate}
\end{theorem}

\begin{proof}
In view of Theorem~\ref{thm:kahlerturan} and the recurrence coefficients of the sieved polynomials, it is left to show that
\begin{equation}\label{eq:firstalternativecorolly}
0\leq c_n(a_{n+1}-c_{n+1})\leq\frac{1}{2}(a_n-c_{n+1})\leq(a_n-c_n)c_{n+1}
\end{equation}
or
\begin{equation}\label{eq:secondalternativecorollary}
0\leq c_n(c_{n+1}-a_{n+1})\leq\frac{1}{2}(c_{n+1}-a_n)\leq(c_n-a_n)c_{n+1}
\end{equation}
for every $n\in\mathbb{N}$. In fact, if $(c_n)_{n\in\mathbb{N}}$ is as in (i), then \eqref{eq:firstalternativecorolly} is always valid: the first inequality is clear, and the two later inequalities can be seen from the computations
\begin{align*}
&(1-c_n)(a_n-c_{n+1}-2c_n(a_{n+1}-c_{n+1}))\\
&=((3-4c_n)(3c_n-1)+2(2c_n-1)^2)c_{n+1}-(1-c_n)(3c_n-1)\\
&\geq(3-4c_n)(3c_n-1)c_{n+1}-(1-c_n)(3c_n-1)\\
&\geq0
\end{align*}
and $2(a_n-c_n)c_{n+1}-(a_n-c_{n+1})=(3-4c_n)c_{n+1}-(1-c_n)\geq0$ for all $n\in\mathbb{N}$. Moreover, if $(c_n)_{n\in\mathbb{N}}$ is as in (ii), then \eqref{eq:secondalternativecorollary} is always fulfilled: the first inequality is clear again, and one has $c_{n+1}-a_n-2c_n(c_{n+1}-a_{n+1})=3c_n-1-(4c_n-1)c_{n+1}\geq0$ and
\begin{align*}
&(3c_n-1)(2(c_n-a_n)c_{n+1}-(c_{n+1}-a_n))\\
&=(-(4c_n-1)(1-c_n)+2(2c_n-1)^2)c_{n+1}+(3c_n-1)(1-c_n)\\
&\geq-(4c_n-1)(1-c_n)c_{n+1}+(3c_n-1)(1-c_n)\\
&\geq0
\end{align*}
for every $n\in\mathbb{N}$.
\end{proof}

\begin{remark}
Note that if $(c_n)_{n\in\mathbb{N}}$ is as in Theorem~\ref{thm:2sieved} (i), then $(c_n)_{n\in\mathbb{N}}$ is particularly nondecreasing, whereas if $(c_n)_{n\in\mathbb{N}}$ is as in Theorem~\ref{thm:2sieved} (ii), then $(c_n)_{n\in\mathbb{N}}$ is nonincreasing (therefore, Theorem~\ref{thm:szwarcturan} yields that also $(P_n(x))_{n\in\mathbb{N}_0}$ satisfies Tur\'{a}n's inequality). One might ask whether Tur\'{a}n's inequality for $(P_n(x;2))_{n\in\mathbb{N}_0}$ is generally true under the more general conditions (i') ``$(c_n)_{n\in\mathbb{N}}\subseteq[1/3,1/2]$ is nondecreasing'' or (ii') ``$(c_n)_{n\in\mathbb{N}}\subseteq[1/2,1)$ is nonincreasing''. However, for both (i') and (ii') there are counterexamples: if $c_n\equiv1/3$, then $\Delta_4(19/20;2)\approx-0.003$, and if $c_n\equiv4/5$, then $\Delta_4(9/10;2)\approx-0.632$.
\end{remark}

Having studied $(T_n^{(\alpha,0)}(x))_{n\in\mathbb{N}_0}$ in Corollary~\ref{cor:kahlerturanmod}, we mention that the polynomials $(T_n^{(\alpha,\alpha)}(x))_{n\in\mathbb{N}_0}$ are just the $2$-sieved ultraspherical polynomials, which can easily be seen from \eqref{eq:genchebrec}. These polynomials were extensively studied in \cite[Section 2]{AAA84}. Moreover, recently certain expansions of these polynomials were studied in \cite{Ka23b}. Theorem~\ref{thm:2sieved} applies to this example.

\section{Further examples}\label{sec:furtherexamples}

In this last section, we turn back to our general criterion Theorem~\ref{thm:kahlerturan} another time and study three further examples $(P_n(x))_{n\in\mathbb{N}_0}$ which fit in our general setting as described in Section~\ref{sec:intro}. Example~\ref{example:first} shall demonstrate the applicability of Theorem~\ref{thm:kahlerturan} besides the class of generalized Chebyshev polynomials. Example~\ref{example:second} and Example~\ref{example:third} shall demonstrate the limits of the applicability and ideas of further possible generalizations.

\begin{example}\label{example:first}
Let $c_2\geq c_1/(1+c_1)$ and $c_n=1/2$ for $n\geq3$. It is easy to see from the recurrence relation \eqref{eq:threetermrec} that $\Delta_n(x)=\Delta_3(x)$ for all $n\geq3$. Using the notation of Theorem~\ref{thm:kahlerturan}, we compute $A_1=0$, $B_1=(1/2-c_1)c_2$, $C_1=1/2-c_1$, $A_2=0$, $B_2=(1/2-c_2)/2$, $C_2=1/2-c_2$ and $A_n=B_n=C_n=0$ for $n\geq3$. Hence, the conditions of Theorem~\ref{thm:kahlerturan} are satisfied. However, the conditions of Theorem~\ref{thm:szwarcturan} need not be satisfied (even if $c_2>c_1/(1+c_1)$).
\end{example}

\begin{example}\label{example:second}
Let $c_2\geq c_1/(1+c_1)$ and $c_n=c_2$ for $n\geq3$. It is easy to see from the recurrence relation \eqref{eq:threetermrec} that $\Delta_n(x)=(c_2/a_2)^{n-2}\Delta_2(x)\;(n\geq2)$. Therefore, Tur\'{a}n's inequality is satisfied, and if $c_2>c_1/(1+c_1)$, then for every $n\in\mathbb{N}$ there exists a constant $K_n>0$ such that $\Delta_n(x)\geq K_n\cdot(1-x^2)$ for all $x\in[-1,1]$. Using the notation of Theorem~\ref{thm:kahlerturan}, we compute $A_1=c_1(1-2c_2)$ and $B_1-A_1=a_2(c_2-c_1)$. Therefore, if $c_2$ is sufficiently close to $c_1/(1+c_1)$, then $A_1>0$ but $B_1<A_1$, so the conditions of Theorem~\ref{thm:kahlerturan} are not always fulfilled. Nevertheless, we mention that Theorem~\ref{thm:kahlerturan} can be modified in such a way that the example is covered (put reasonable conditions on $\Delta_3(x)$ and conclude from $\Delta_n(x)$ to $\Delta_{n+2}(x)$ only for $n\geq2$).
\end{example}

\begin{example}\label{example:third}
Let $c_n=2n/(4n+3)$ if $3|n$ and $c_n=1/2$ else ($(P_n(x))_{n\in\mathbb{N}_0}$ corresponds to the ``$3$-sieved ultraspherical polynomials'' for $\alpha=-1/4$). Using the recurrence relation \eqref{eq:threetermrec}, one can establish the nonnegative representations $\Delta_{3n-2}(x)=(P_{3n-1}(x)-x P_{3n-2}(x))^2+(1-x^2)P_{3n-2}^2(x)$, $\Delta_{3n-1}(x)=(P_{3n}(x)-x P_{3n-1}(x))^2+(1-x^2)P_{3n-1}^2(x)$, $\Delta_{3n}(x)=(n-1)!/((3/2)_n(x^2+1))\cdot\sum_{k=0}^{n-1}(3/2)_k/k!\cdot[((x^2+1)P_{3k}(x)-2x^3P_{3k+1}(x))^2+(1-x^2)^2P_{3k+1}^2(x)]$ for all $n\in\mathbb{N}$. Since $A_{3n-2}=1/(8n+2)$ and $C_{3n-2}=0$ for all $n\in\mathbb{N}$ (with the notation of Theorem~\ref{thm:kahlerturan}), the conditions of Theorem~\ref{thm:kahlerturan} are not fulfilled. In contrast to Example~\ref{example:second}, this cannot be overcome by just adjusting the initial conditions and so on. However, the example shows how conclusions from $\Delta_n(x)$ to $\Delta_{n+k}(x)$ for $k\in\mathbb{N}\backslash\{1,2\}$ (which helped us to find the formula for $\Delta_{3n}(x)$) may work, too. Sieved ultraspherical polynomials were also studied in \cite{BP01}, and the arguments of the reference are related to ours.
\end{example}

We finally mention that Theorem~\ref{thm:kahlerturan2} seems to be less suitable for further classes because explicit computations of $(a_{m,n})_{n\in\mathbb{N}}$ and $(c_{m,n})_{n\in\mathbb{N}}$ will generally be difficult, even if $(a_n)_{n\in\mathbb{N}}$ and $(c_n)_{n\in\mathbb{N}}$ are known. However, we emphasize that Theorem~\ref{thm:kahlerturan2} might be helpful as soon as one studies a class of orthogonal polynomials whose orthogonalization measures are of the form $(1-x^2)^{\alpha}f(x)\chi_{(-1,1)}(x)\,\mathrm{d}x$ ($f$ independent of $\alpha$) with explicitly known recurrence coefficients $a_n(\alpha)$ and $c_n(\alpha)$ because then $a_{m,n}(\alpha)=a_n(m+\alpha)$ and $c_{m,n}(\alpha)=c_n(m+\alpha)$ (like for generalized Chebyshev polynomials in Section~\ref{sec:application}). Moreover, the recurrence relation \eqref{eq:astrecm} and the relation to chain sequences might be helpful for further classes, too.

\bibliography{TuranGeneralGenCheb}
\bibliographystyle{amsplain}

\end{document}